\documentclass[reqno]{amsart}

\usepackage{times,fancyhdr}
\usepackage[dvips]{graphicx}
\usepackage{amssymb}
\usepackage{amsmath}
\usepackage{amsfonts}
\usepackage{hyperref}

\def \ce{\centering}
\def \R{\mathbb{R}}
\def \P{\mathbb{P}}
\def \N{\mathbb{N}}
\def \E{\mathbb{E}}
\def \D{\mathbb{D}}
\def \T{\mathbb{T}}
\def \H{\mathbb{H}}
\def \L{\mathbb{L}}
\def \C{\mathbb{C}}
\def \Z{\mathbb{Z}}
\def \bf{\textbf}
\def \it{\textit}
\def \sc{\textsc}
\def \bop {\noindent\textbf{Proof }}
\def \eop {\hbox{}\nobreak\hfill
\vrule width 2mm height 2mm depth 0mm
\par \goodbreak \smallskip}
\def \ni {\noindent}
\def \sni {\ss\ni}
\def \bni {\bigskip\ni}
\def \ss {\smallskip}
\def \F{\mathcal{F}}
\def \g{\mathcal{g}}
\def \K{\mathcal{K}}
\vspace{9mm}
\begin{document}

\title[  Controllability  results   ]
{Controllability of fractional neutral  functional differential
equations driven by fractional Brownian motion  with infinite delay}

\author[   Lakhel    ]{  El Hassan  Lakhel  }
\maketitle
\begin{center}{ Cadi Ayyad University,  National School of
Applied Sciences,  46000 Safi, Morocco}\end{center}

 \vspace{0.8cm}
  \footnotetext[1]{Lakhel E.: e.lakhel@uca.ma (Corresponding author)}
\noindent \textbf{Abstract. } In this paper we study the
controllability  of   fractional  neutral stochastic functional
differential equations with infinite delay driven  by    fractional
Brownian motion in a real separable Hilbert space.
  The controllability results are obtained by using stochastic analysis and  a fixed-point strategy. Finally, an illustrative example
is provided to demonstrate the effectiveness of the theoretical
result.\\




\vspace{.08in} \noindent {\textbf AMS Subject Classification:}
35R10, 93B05  60G22, 60H20.

\def \ce{\centering}
\def \R{\mathbb{R}}
\def \P{\mathbb{P}}
\def \N{\mathbb{N}}
\def \E{\mathbb{E}}
\def \D{\mathbb{D}}
\def \T{\mathbb{T}}
\def \H{\mathbb{H}}
\def \L{\mathbb{L}}
\def \C{\mathbb{C}}
\def \Z{\mathbb{Z}}
\def \bf{\textbf}
\def \it{\textit}
\def \sc{\textsc}
\def \bop {\noindent\textbf{Proof }}
\def \eop {\hbox{}\nobreak\hfill
\vrule width 2mm height 2mm depth 0mm
\par \goodbreak \smallskip}
\def \ni {\noindent}
\def \sni {\ss\ni}
\def \bni {\bigskip\ni}
\def \ss {\smallskip}
\def \F{\mathcal{F}}
\def \g{\mathcal{g}}
\def \K{\mathcal{K}}
 \numberwithin{equation}{section}
\newtheorem{theorem}{Theorem}[section]
\newtheorem{lemma}[theorem]{Lemma}
\newtheorem{proposition}[theorem]{Proposition}
\newtheorem{definition}[theorem]{Definition}
\newtheorem{example}[theorem]{Example}
\newtheorem{remark}[theorem]{Remark}
\allowdisplaybreaks

\section{Introduction}
Fractional Brownian motion (fBm)  $\{B^H(t):t \in \R\}$ is a
Gaussian stochastic process, which depends on a parameter
$H\in(0,1)$ called Hurst index, for additional details on the
fractional Brownian motion, we refer the reader to \cite{MV68}. This
stochastic  process has self-similarity, stationary increments, and
long-range dependence properties.  It is known that fractional
Brownian motion is a generalization of Brownian motion and it
reduces to a standard Brownian motion when   $H=\frac{1}{2}$.
Fractional Brownian motion is not a semimartingale if
$H\neq\frac{1}{2}$ (see Biagini \it{al.} \cite{biagini08}), the
classical It\^{o} theory cannot be used to construct a stochastic
calculus with respect to fBm.

Fractional differential equations have recently been proved to be
valuable tools in the modeling of many phenomena in various fields
of  physics, finance,   electrical engineering, telecommunication
networks, and so on.    There has been a significant development in
fractional differential equations.   Some authors have considered
fractional stochastic       equations, we refer to Ahmed
\cite{ahm09}, El-Bori
      \cite{bor06},  Cui and Yan \cite{cui12}, Sakthivel et al.
      \cite{sak13a,sak13b}. The perturbed terms of these fractional
      equations are Wiener processes. For more details, one can see
the monographs of Kilbas et al. \cite{kil06},       Zhou
\cite{zho14},  and  Zhou et al. \cite{zho12}    and the references
      therein.

 In many areas of science, there has been an
increasing interest in the investigation of the systems
incorporating memory or aftereffect, i.e., there is the effect of
delay on state equations. Therefore, there is a real need to discuss
stochastic evolution systems with delay. In many mathematical
models the claims often display long-range memories, possibly due to
extreme weather, natural disasters, in some cases, many stochastic
dynamical systems depend not only on present and past states, but
also contain the derivatives with delays. Neutral functional
differential equations are often used to describe such systems.

Moreover, control theory  is an area of  application-oriented
mathematics which deals  with basic principles underlying the
analysis and design of control systems. Roughly speaking,
controllability generally means that it is possible to steer a
dynamical control system from an arbitrary initial state to an
arbitrary final state using the set of admissible controls.
Controllability plays a crucial role in a lot of control problems,
such as the case of stabilization of unstable systems by feedback or
optimal control \cite{klam07,klam13}. The controllability concept
has been studied extensively in the fields of finite-dimensional
systems, infinite-dimensional systems, hybrid systems, and
behavioral systems.  If a system cannot be controlled completely
then different types of controllability can be defined such as
approximate, null, local null and local approximate null
controllability. For more details the reader may refer to
\cite{klam13,ren11,ren13a} and the references therein. In this
paper, we study the controllability  of fractional neutral
functional stochastic differential equations   of the form
 {
\begin{equation}\label{eq1}
 \left\{\begin{array}{ll}
 &d[J^{1-\alpha}_t(x(t)-g(t,x_t)-\varphi(0)+g(0,\varphi))]=[Ax(t)+f(t,x_t)+Bu(t)]dt\\\\
 &\qquad\qquad\qquad\qquad\qquad\qquad\qquad\qquad\qquad\quad+\sigma (t)dB^H(t),\, t
\in[0,
T],\\\\
 &x(t)=\varphi(t)\in L^2(\Omega,\mathcal{B}_h),\; for\; a.e.\;
t\in(-\infty,0],
\end{array}\right.
\end{equation}
where $\frac12<\alpha<1,$  $J^{1-\alpha}$ is the $(1-\alpha)-$order
Riemann-Liouville  fractional integral operator,  $A$ is the
infinitesimal generator of an analytic semigroup of bounded linear
operators, $(S(t))_{t\geq 0}$, in a Hilbert space $X$;  $B^H$ is a
fractional Brownian motion with $H>\frac{1}{2} $ on a real and
separable Hilbert space $Y$; and the control function $u(\cdot)$
takes values in $L^2([0,T],U)$, the Hilbert space of admissible
control functions for a separable Hilbert space  $U$; and $B$ is a
bounded linear operator from $U$ into $X$.

\pagestyle{fancy} \fancyhead{} \fancyhead[EC]{E. LAKHEL}
\fancyhead[EL,OR]{\thepage} \fancyhead[OC]{ Controllability
results} \fancyfoot{}
\renewcommand\headrulewidth{0.5pt}

The history $x_{t}:(-\infty,0]\to X$, $x_{t}(\theta)=x(t+\theta)$,
belongs to an abstract phase  space ${\mathcal{B}_h}$ defined
axiomatically,   and $ f,g:[0,T]\times \mathcal{B}_h \to X $, and
$\sigma:[0,T] \rightarrow \mathcal{L}_2^0(Y,X)$, are appropriate
functions, where $\mathcal{L}_2^0(Y,X)$ denotes the space of all
$Q$-Hilbert-Schmidt operators from $Y$ into $X$ (see section 2
below).


For potential applications in telecommunications networks, finance
markets, biology and other fields \cite{iRL,iLL}, stochastic
differential equations driven by fractional Brownian motion have
attracted researcher's  great interest. Especially, we mention here
the recent papers  \cite{lak16,lak15,lak8,ren13}. Moreover, Dung
studied the existence and uniqueness of impulsive stochastic
      Volterra integro-differential equation driven by fBm in
     \cite{dung15} .  Using  the Riemann-Stieltjes integral,
      Boufoussi  et al.  \cite{boufoussi2} proved the existence and
    uniqueness of a mild solution to a related problem and studied the dependence of the
    solution on the initial condition in   infinite dimensional space.
    More
    recently, Li \cite{li13} investigated the existence  of mild solution to a class of stochastic delay fractional evolution equations driven by fBm.
      Caraballo et al.   \cite{carab}, and Boufoussi
     and Hajji \cite{boufoussi3} have discussed the existence,
     uniqueness and exponential asymptotic behavior
     of mild solutions by using the Wiener integral.

To the best of the author's knowledge, an investigation concerning
the controllability for fractional  neutral stochastic differential
equations with infinite delay of the form \eqref{eq1} driven by a
fractional Brownian motion has not yet been conducted. Thus, we will
make the first attempt to study such problem in this paper. Our
results are motivated by those in \cite{lak16,lak8} where the
controllability of mild solutions to  neutral stochastic functional
integro-differential equations driven by fractional Brownian motion
with finite delays   are studied.

The  outline  of this paper is  as follows:  In the next section,
some necessary notations and concepts are provided. In Section 3, we
derive the controllability of fractional neutral
 stochastic differential systems driven by a fractional Brownian motion.
Finally, in Section 4, we conclude with an example to illustrate the
applicability of the general theory.

\section{Preliminaries}

We collect some notions, concepts and lemmas concerning the Wiener
integral with respect to an infinite dimensional fractional
Brownian,  and we recall some basic results  which will be used
throughout the whole of this paper.

Let $(\Omega,\mathcal{F}, \mathbb{P})$ be a complete probability
space. A standard fractional Brownian motion (fBm) $\{\beta^H(t),
t\in\mathbb{R}\}$
 with Hurst parameter $H\in (0, 1)$ is a zero mean  Gaussian process with continuous
sample paths such that
\begin{eqnarray}\label{A3}
R_H(t ,
s)=\mathbb{E}[\beta^H(t)\beta^H(s)]=\frac{1}{2}\big(t^{2H}+s^{2H}-|t-s|^{2H}\big),\qquad
  \qquad s, t\in\mathbb{R}.
\end{eqnarray}


Let $X$ and $Y$ be two real, separable Hilbert spaces and let
$\mathcal{L}(Y,X)$ be the space of bounded linear operator from $Y$
to $X$. For the sake of convenience, we shall use the same notation
to denote the norms in $X,Y$ and $\mathcal{L}(Y,X)$. Let $Q\in
\mathcal{L}(Y,Y)$ be an operator defined by $Qe_n=\lambda_n e_n$
with finite trace
 $trQ=\sum_{n=1}^{\infty}\lambda_n<\infty$. where $\lambda_n \geq 0 \; (n=1,2...)$ are non-negative
  real numbers and $\{e_n\}\;(n=1,2...)$ is a complete orthonormal basis in $Y$.

We define the infinite dimensional fBm on $Y$ with covariance $Q$ as

 $$
 B^H(t)=B^H_Q(t)=\sum_{n=1}^{\infty}\sqrt{\lambda_n}e_n\beta_n^H(t),
 $$
 where $\beta_n^H$ are real, independent fBm's. This process is  Gaussian, it
 starts from $0$, has zero mean and covariance:
 $$E\langle B^H(t),x\rangle\langle B^H(s),y\rangle=R(s,t)\langle Q(x),y\rangle \;\; \mbox{for all}\; x,y \in Y \;\mbox {and}\;  t,s \in [0,T]$$
In order to define Wiener integrals with respect to the $Q$-fBm, we
introduce the space $\mathcal{L}_2^0:=\mathcal{L}_2^0(Y,X)$  of all
$Q$-Hilbert-Schmidt operators $\psi:Y\rightarrow X$. We recall that
$\psi \in \mathcal{L}(Y,X)$ is called a $Q$-Hilbert-Schmidt
operator, if
$$  \|\psi\|_{\mathcal{L}_2^0}^2:=\sum_{n=1}^{\infty}\|\sqrt{\lambda_n}\psi e_n\|^2 <\infty,
$$
and that the space $\mathcal{L}_2^0$ equipped with the inner product
$\langle \varphi,\psi
\rangle_{\mathcal{L}_2^0}=\sum_{n=1}^{\infty}\langle
\varphi e_n,\psi e_n\rangle$ is a separable Hilbert space.\\
Let $\phi(s);\,s\in [0,T]$ be a function with values in
$\mathcal{L}_2^0(Y,X)$, such that $\sum_{n=1}^{\infty}\|K^*\phi
Q^{\frac{1}{2}}e_n\|_{\mathcal{L}_2^0}^2<\infty.$ The Wiener
integral of $\phi$ with respect to $B^H$ is defined by

\begin{equation}\label{int}
\int_0^t\phi(s)dB^H(s)=\sum_{n=1}^{\infty}\int_0^t
\sqrt{\lambda_n}\phi(s)e_nd\beta^H_n(s).
\end{equation}

Now, we end this subsection by stating the following result which is
fundamental to prove our result. It can be proved by  similar
arguments as those used to prove   Lemma 2 in \cite{carab}.
\begin{lemma}\label{lem2}
If $\psi:[0,T]\rightarrow \mathcal{L}_2^0(Y,X)$ satisfies $\int_0^T
\|\psi(s)\|^2_{\mathcal{L}_2^0}ds<\infty$,
 then the above sum in $(\ref{int})$ is well defined as a $X$-valued random variable and
 we have$$ \mathbb{E}\|\int_0^t\psi(s)dB^H(s)\|^2\leq 2Ht^{2H-1}\int_0^t \|\psi(s)\|_{\mathcal{L}_2^0}^2ds.
 $$
\end{lemma}

It is known that the study of theory of  differential equation with
infinite delays depends on a choice of the abstract phase space. We
assume that the phase space $\mathcal{B}_h$ is a linear space of
functions mapping $(-\infty ,0]$ into $X$, endowed with a norm
$\|.\|_{\mathcal{B}_h}$.  We shall introduce some basic definitions,
notations and lemma which are used in this paper. First, we present
the abstract phase space $\mathcal{B}_h$. Assume that $h:(-\infty
,0]\longrightarrow [0,+\infty)$  is a continuous function with
$l=\int_{-\infty}^0h(s)ds<+\infty$.

We define the abstract phase space $\mathcal{B}_h$ by
$$
\begin{array}{ll}
                    \mathcal{B}_h=   & \{\psi:(-\infty,0]\longrightarrow X \text{ for any } \tau>0, (\E\|\psi\|^2)^{\frac{1}{2}} \text{ is bounded and measurable }  \\
                       & \text{ function on } [-\tau,0] \text{ and } \int_{-\infty}^0h(t)\sup_{t\leq s\leq0}
                       (\E\|\psi(s)\|^2)^{\frac{1}{2}}dt<+\infty\}.
                    \end{array}
$$
If we equip this space with the norm

$$
\|\psi\|_{\mathcal{B}_h}:=\int_{-\infty}^0h(t)\sup_{t\leq
s\leq0}(\E\|\psi(s)\|^2)^{\frac{1}{2}}dt,
$$
then it is clear that $(\mathcal{B}_h,\|.\|_{\mathcal{B}_h})$  is a
Banach space.

Next, We  consider the space $\mathcal{B}_{T}$, given by

$$
    \mathcal{B}_{T} = \{ x: x\in \mathcal{C}((-\infty,T],X),  \text{  with  } \,  x_0=\varphi\in \mathcal{B}_h
    \},
$$
where $\mathcal{C}((-\infty,T],X)$ denotes the space of all
continuous $X-$valued stochastic processes $\{x(t),\;
t\in(-\infty,T]\}$.  The function $\|.\|_{\mathcal{B}_{T}}$ to be a
semi-norm in $\mathcal{B}_{T}$, it is defined by
$$
\|x\|_{\mathcal{B}_{T}}=\|x_0\|_{\mathcal{B}_h}+\sup_{0\leq t \leq
T}(\E\|x(t)\|^2)^{\frac{1}{2}}.
$$
The following lemma  is a common property of phase spaces.
\begin{lemma}\cite{li-li} Suppose $x\in \mathcal{B}_{T}$, then for all $t\in [0,T]$ ,  $x_t\in
\mathcal{B}_h$  and
$$
l(\E\|x(t)\|^2)^{\frac12}\leq\|x_{t}\|_{\mathcal{B}_h}\leq l\sup_{0
\leq s\leq t}(\E\|x(s)\|^2)^{\frac12}+ \|x_{0 }\|_{\mathcal{B}_h},
$$
where $l=\int_{-\infty}^0h(s)ds<\infty$.

\end{lemma}

Let us give the following well-known definitions related to
fractional order differentiation and integration.
\begin{definition}
The Riemann-Liouville fractional integral of order $\alpha>0$ of a
function $f:\R^+\longrightarrow X$ is defined by
$$
J^\alpha_tf(t)=\frac{1}{\Gamma(\alpha)}\int_0^t\frac{f(s)}{(t-s)^{1-\alpha}}ds,
$$
where $\Gamma(.)$ is the Gamma function.
\end{definition}

\begin{definition}
The Riemann-Liouville fractional derivative of order
$\alpha\in(0,1)$ of a function $f:\R^+\longrightarrow X$ is defined
by
$$
D^\alpha_tf(t)=\frac{d}{dt}J^{1-\alpha}_tf(t).
$$
\end{definition}
\begin{definition} The Caputo fractional derivative of order
$\alpha\in(0,1)$ of $f: \R^+\longrightarrow X$ is defined by
$$
^CD^\alpha_tf(t)=D^\alpha_t(f(t)-f(0)).
$$

\end{definition}
For more details on fractional calculus, one can see \cite{kil06}.

 We suppose  $0 \in \rho(A)$, the resolvent set of $A$, and the semigroup, $(S
(t))_{t\geq 0}$, is uniformly  bounded. That is,  there exists $M
\geq 1$  such that  $\|S (t)\|\leq M $ for every $t \geq 0$. Then it
is possible to define the fractional power $(-A)^{\alpha}$ for $0 <
\alpha \leq 1$, as a closed linear operator on its domain
$D(-A)^{\alpha}$. Furthermore, the subspace $D(-A)^{\alpha}$ is
dense in $X$, and the expression
$$\|h\|_{\alpha} =\|(-A)^{\alpha}h\|$$ defines a norm in
$D(-A)^{\alpha}$. If $X_{\alpha}$ represents the space
$D(-A)^{\alpha}$ endowed with the norm $\|.\|_{\alpha}$, then the
following properties hold (see \cite{pazy}, p. 74).
\begin{lemma}\label{lem3} Suppose that $A, X_{\alpha},$ and $(-A)^{\alpha}$ are as described above.
\begin{itemize}
\item[(i)] For \,$0<\alpha \leq  1$,  $X_{\alpha}$ is a Banach space.
\item[(ii)] If\,  $0 <\beta \leq \alpha, $ then the injection $X_{\alpha}
\hookrightarrow X_{\beta}$ is continuous.
 \item[(iii)] For every \,$0<\alpha \leq 1,$ there exists $M_{\alpha} > 0 $ such that
 $$\| (-A)^{\alpha}S (t)\|\leq M_{\alpha}t^{-\alpha}e^{-\lambda t}
 , \;\;\;\;t>0,\;\; \lambda>0 .$$
 \end{itemize}
 \end{lemma}
\section{Controllability Result}

  Before starting and proving our main result, we introduce the concepts
of a mild solution of the problem (\ref{eq1})  and the meaning of
controllability of fractional  neutral stochastic functional
differential equation.

\begin{definition}
An $X$-valued  process $\{x(t) : t\in(-\infty,T]\}$ is a mild
solution of (\ref{eq1}) if

\begin{enumerate}
\item  $x(t)$ is continuous on $[0,T]$ almost surely and for each $s\in[0,t)$ and $\alpha\in (0,1)$ the function  $(t-s)^{\alpha-1}AS_\alpha(t-s)g(s,x_s)$ is
 integrable,
\item for arbitrary $t \in [0,T]$, we have

\begin{equation}\label{eqmild2}
\begin{array}{ll}
x(t)&=T_\alpha(t)(\varphi(0)-g(0,\varphi ))+g(t,x_t)\\ \\
&+ \int_0^t (t-s)^{\alpha-1}AS_\alpha(t-s)g(s,x_s)ds +\int_0^t(t-s)^{\alpha-1} S_\alpha(t-s)f(s,x_s)ds\\ \\
&+\int_0^t (t-s)^{\alpha-1}S_\alpha(t-s)Bu(s)ds+\int_0^t
(t-s)^{\alpha-1}S_\alpha(t-s)\sigma(s)dB^H(s),\; \mathbb{P}-a.s.
\end{array}
\end{equation}
\item $x(t)=\varphi(t)$ on $(-\infty,0]$ satisfying $\|\varphi\|_{\mathcal{B}_h}^2<\infty,$
\end{enumerate}
where
$$
T_\alpha(t)x=\int_0^\infty\eta_\alpha(\theta)S(t^\alpha\theta)xd\theta,\;t\geq0,\;
x\in X.
$$
$$
S_\alpha(t)x=\alpha\int_0^\infty\theta\eta_\alpha(\theta)S(t^\alpha\theta)xd\theta,\;t\geq0,\;
x\in X,
$$
where

$$
\eta_\alpha(\theta)=\frac{1}{\alpha}\theta^{-1-\frac{1}{\alpha}}\omega_\alpha(\theta^{-\frac{1}{\alpha}})\geq0,
$$

$$
\omega_\alpha(\theta)=\frac{1}{\pi}\sum_{n=1}^\infty(-1)^{n-1}\theta^{-\alpha
n-1}\frac{\Gamma(n\alpha+1)}{n!}\sin(n\alpha\pi),\quad
\theta\in]0,\infty[,
$$

 $\eta_\alpha$ is a probability density function defined on
$(0,\infty)$.

\end{definition}
\begin{remark} (see \cite{zho10})

\begin{equation}\label{prob1}
\int_0^\infty
\theta\eta_\alpha(\theta)d\theta=\frac{1}{\Gamma(1+\alpha)}.
\end{equation}
\end{remark}
The following properties of $T_\alpha$ and $S_\alpha$ appeared in
\cite{zho10} are useful.
\begin{lemma}\label{operators}
Under the previous assumptions on $S(t)$, $t\geq0$ and $A$,   the
operators  $T_\alpha(t)$ and $S_\alpha(t)$ have the following
properties:
\begin{itemize}
  \item [(i)] For any $x\in X$, $\|T_\alpha(t)x\|\leq M\|x\|$,  $\|S_\alpha(t)x\|\leq \frac{M}{\Gamma(\alpha)}\|x\|$.
  \item [(ii)] $\{T_\alpha(t),\;t\geq0\}$ and
  $\{S_\alpha(t),\;t\geq0\}$ are strongly continuous.
  \item [(iii)] For any $t>0$, \, $T_\alpha(t)$ and $S_\alpha(t)$ are
  also compact operators if $S(t)$ is compact.
  \item [(iv)] For any $x\in X$,  $\beta\in(0,1)$ and $\delta\in(0,1],$ we have
  $$AS_\alpha(t)x=A^{1-\beta}S_\alpha A^\beta x, \text{ and }
  \|A^\delta S_\alpha(t)\|\leq\frac{\alpha M_{\delta}}{t^{\alpha
  \delta}} \frac{\Gamma(2-\delta)}{\Gamma(1+\alpha(1-\delta))},\; t\in (0,T].
  $$
\end{itemize}

\end{lemma}
\begin{definition}
The fractional  neutral  stochastic functional differential equation
(\ref{eq1}) is said to be controllable on the interval $(-\infty,T]$
if for every initial stochastic process $\varphi$ defined on
$(-\infty,0]$, there exists a stochastic control $u\in L^2([0,T],
U)$ such that the mild solution $x(\cdot)$ of  (\ref{eq1}) satisfies
$x(T)=x_1$, where $x_1$ and $T$ are the preassigned terminal state
and time, respectively.
\end{definition}

Our main result in this paper is based on the following fixed point
theorem.
\begin{theorem}(Karasnoselskii's fixed point theorem)
Let $V$ be a bounded closed and convex subset of  a Banach space $X$
 and  let $\Pi_1$, $\Pi_2$ be two operators of $V$
 into $X$ satisfying:
\begin{enumerate}
\item  $\Pi_1(x)+\Pi_2(x)\in V$ whenever $x\in V$,

  \item  $\Pi_1$ is a contraction mapping, and
  \item $\Pi_2$ is  completely continuous.
\end{enumerate}
Then,  there exists a $z \in V$ such that $z=\Pi_1(z)+\Pi_2(z)$.
\end{theorem}

 In order to establish the controllability of
(\ref{eq1}), we impose the following conditions on the data of the
problem:

\begin{itemize}
\item [$(\mathcal{H}.1)$]
 The  analytic semigroup, $(S (t))_{t\geq
0}$,  generated by $A$  is compact for $t>0$, and    there exists
$M\geq1$ such that
$$ \sup_{t\geq0}\| S (t)\|\leq M, \qquad \text{ and }   c_1= \|(-A)^{-\beta}\|.
$$
\item [$(\mathcal{H}.2)$] The map $f:[0,T]\times \mathcal{B}_h\to X$
satisfies the following conditions:
\begin{enumerate}
  \item[(i)] The function  $t\longmapsto f(t,x)$ is measurable for each $x\in
  \mathcal{B}_h$,  the function
   $x\longmapsto f(t,x)$ is continuous for almost all $t\in
  [0,T]$,
  \item[(ii)]there exists a nonnegative function $p\in
  L^1([0,T],\R^+)$, and a continuous nondecreasing function $\vartheta: \R^+\longrightarrow (0,+\infty)$ such
  that for $\delta>\frac{1}{2\alpha-1}$, \; $(\alpha\in(\frac12,1)),$
  $$
\int_0^T(\vartheta(s))^\delta ds<\infty,\qquad
\liminf_{k\longrightarrow+\infty}\frac{\vartheta(k)}{k}=\gamma<\infty,
  $$
  and
  $$
\|f(t,x)\|^2\leq p(t)\vartheta(\|x\|_{\mathcal{B}_h}^2), \,\text{
for all }  x\in \mathcal{B}_h\, , \,\text{ almost surely and for
a.e. }t\in[0,T].
  $$

\end{enumerate}

\item [$(\mathcal{H}.3)$] The function $g:[0,T]\times\mathcal{B}_h\longrightarrow X$ is
continuous. For   $\beta\in(0, 1), $ satisfied with
$\alpha\beta>\frac{1}{2}$,\;  the function $g$ is $X_{\beta}$-valued
and there exists positive  constant   $M_g,$ such
 that
 $$\|(-A)^{\beta}g(t,x)-(-A)^{\beta}g(t,y)\|^2 \leq M_g \|x-y\|^2_{\mathcal{B}_{h}}, \,\text{
for all }  x\in \mathcal{B}_h\, , \,\text{ almost surely and for
a.e. }t\in[0,T],
 $$

$$
\|(-A)^{\beta}g(t,x)\|^2\leq M_g[\|x\|^2_{\mathcal{B}_{h}}+1],
\,\text{ for all }  x\in \mathcal{B}_h\, , \,\text{ almost surely
and for a.e. }t\in[0,T].
$$

\item [$(\mathcal{H}.4)$] There exists a constant $p>\frac{1}{2\alpha-1}$ such that  the function $\sigma:[0,\infty)\rightarrow \mathcal{L}_2^0(Y,X)$
satisfies
 $$\int_0^T\|\sigma(s)\|^{2p}_{\mathcal{L}_2^0}ds< \infty,\;\; \forall T>0 .
 $$
 \item [$(\mathcal{H}.5)$] The linear operator $W$ from $U$ into $X$
 defined by
 $$
Wu=\int_0^T(T-s)^{\alpha-1}S_\alpha(T-s)Bu(s)ds
 $$
 has an inverse operator $W^{-1}$ that takes values in $L^2([0,T],U)\setminus ker
 W$, where $$
 ker
 W=\{x\in L^2([0,T],U):  \; W x=0\}
 $$ (see \cite{klam07}), and there
 exists finite
 positive constants $M_b,$ $M_w$ such that $\|B\|^2\leq M_b$ and $\|W^{-1}\|^2\leq M_w.$
\item [$(\mathcal{H}.6)$] Assume the following inequality holds:
{\small{
 \begin{equation}\label{cond1}
 \begin{array}{ll}
    &  24l^2 \{[c_1^2+\frac{
   T^{2\alpha\beta}\alpha^2M^2_{1-\beta}\Gamma^2(\beta+1)}{(2\alpha\beta-1)\Gamma^2(\alpha\beta+1)}]M_g+\gamma(1+\frac{6M^2M_bM_wT^{2\alpha}}{(2\alpha-1)\Gamma^2(\alpha)})
\frac{M^2T}{\Gamma^2(\alpha)}\int_0^T(T-s)^{2\alpha-2}p(s)
ds\\\\
&+\frac{6M^2M_bM_wT^{2\alpha}}{(2\alpha-1)\Gamma^2(\alpha)}[c_1^2
+\frac{\alpha^2M_{1-\beta}^2T^{2\alpha\beta}\Gamma^2(\beta+1)}{(2\alpha\beta-1)\Gamma^2(\alpha\beta+1)}]M_g\}<1.
 \end{array}
\end{equation}}}

\end{itemize}

The main result of this chapter is the following.
\begin{theorem}\label{th1}
Suppose that $(\mathcal{H}.1)-(\mathcal{H}.6)$ hold. Then,  the
system  (\ref{eq1}) is controllable  on $(-\infty,T]$.

\end{theorem}
\begin{proof} Transform the problem(\ref{eq1}) into a fixed-point
problem. To do this, using the hypothesis $(\mathcal{H}.5)$ for an
arbitrary function $x(\cdot)$, define the control by

\begin{equation}\label{control1}
\begin{array}{lll}
  u(t) & =&W^{-1}\{x_1-T_\alpha(T)[\varphi(0)-g(0,x_0)]-g(T,x_T))\\\\
   & -&\int_0^T (T-s)^{\alpha-1}AS_\alpha(T-s)g(s,x_s)ds -\int_0^T (T-s)^{\alpha-1}S_\alpha(T-s)f(s,x_s)ds\\\\
   & -& \int_0^T(T-s)^{\alpha-1}
   S_\alpha(T-s)\sigma(s)dB^H(s)\}(t),\; t\in[0,T].\\
\end{array}
\end{equation}

To formulate the controllability problem in the form suitable for
application of the  fixed point theorem, put the control $u(.)$ into
the stochastic control system (\ref{eqmild2}) and obtain a non
linear  operator $\Pi$ on $\mathcal{B}_{T}$ given  by
$$ \Pi(x)(t)=\left\{
\begin{array}{ll}
& \varphi(t), \;\;\; \text {if } \; t\in(-\infty,0], \\\\
&T_\alpha(t)(\varphi(0)-g(0,\varphi ))+g(t,x_t)+ \int_0^t (t-s)^{\alpha-1}AS_\alpha(t-s)g(s,x_s)ds\\\\
& +\int_0^t
(t-s)^{\alpha-1}S_\alpha(t-s)f(s,x_s)ds+\int_0^t(t-s)^{\alpha-1}
S_\alpha(t-s)Bu(s)ds\\\\
&+\int_0^t(t-s)^{\alpha-1}
S_\alpha(t-s)\sigma(s)dB^H(s),\;\;\; \text {if }\;
 t\in[0,T].
\end{array}\right.
$$

 Then it is clear that to prove the existence of mild
solutions to equation (\ref{eq1}) is equivalent to find a fixed
point for the operator $\Pi$. Clearly, $\Pi x(T)=x_1$, which means
that the control $u$ steers the system from the initial state
$\varphi$ to $x_1$ in time $T$, provided we can obtain a fixed point
of the operator $\Pi$ which implies that the system in controllable.

 Let $y:(-\infty,T]\longrightarrow X$ be the function defined by
 $$
y(t)=\left\{\begin{array}{ll}
       \varphi(t), & \text {if } \; t\in(-\infty,0], \\
       S(t)\varphi(0), & \text {if }\; t\in[0,T],
     \end{array}\right.
 $$
then, $y_0=\varphi$. For each function $z\in \mathcal{B}_{T}$, set
$$
x(t)=z(t)+y(t).
$$

It is obvious that $x$ satisfies the stochastic control system
(\ref{eqmild2}) if and only if $z$ satisfies $z_0=0$ and

\begin{equation}\label{eqmild3}
\begin{array}{ll}
z(t)=&g(t,z_t+y_t)-T_\alpha(t)g(0,\varphi )+
\int_0^t(t-s)^{\alpha-1} AS_\alpha(t-s)g(s,z_s+y_s)ds
\\\\
&+\int_0^t (t-s)^{\alpha-1}S_\alpha(t-s)f(s,z_s+y_s)ds+\int_0^t (t-s)^{\alpha-1}S_\alpha(t-s)Bu_{z+y}(s)ds\\\\
&+\int_0^t (t-s)^{\alpha-1}S_\alpha(t-s)\sigma(s)dB^H(s),\;\;\;

\end{array}
\end{equation}

where $u_{z+y}(t)$ is obtained from (\ref{control1}) by replacing
$x_t=z_t+y_t$.\\
%

Set
$$
\mathcal{B}_{T}^0=\{z\in \mathcal{B}_{T}: z_0=0\};
$$
for any $z\in B_{T}^0$, we have
$$
\|z\|_{\mathcal{B}_{T}^0}=\|z_0\|_{\mathcal{B}_h}+
\sup_{t\in[0,T]}(\E\|z(t)\|^2)^{\frac{1}{2}}=\sup_{t\in[0,T]}(\E\|z(t)\|^2)^{\frac{1}{2}}.
$$
Then, $(\mathcal{B}_{T}^0,\|.\|_{\mathcal{B}_{T}^0})$ is a Banach
space. Define the operator $\widehat{\Pi}:
\mathcal{B}_{T}^0\longrightarrow \mathcal{B}_{T}^0$ by

\begin{equation}
(\widehat{\Pi}z)(t)=\left\{\begin{array}{ll}
&0\;\;\text {if } \; t\in(-\infty,0], \\\\
&g(t,z_t+y_t)-T_\alpha(t)g(0,\varphi ))+ \int_0^t (t-s)^{\alpha-1}AS_\alpha(t-s)g(s,z_s+y_s)ds \\\\
&+\int_0^t (t-s)^{\alpha-1}S_\alpha(t-s)f(s,z_s+y_s)ds\\\\
&+\int_0^t (t-s)^{\alpha-1}S_\alpha(t-s)Bu_{z+y}(s)ds\\\\
&+\int_0^t(t-s)^{\alpha-1}
S_\alpha(t-s)\sigma(s)dB^H(s),\;\;\;\;\;\;\;\text {if } \;
 t\in[0,T].
\end{array}\right.
\end{equation}

Set
$$
\mathcal{B}_k=\{z\in \mathcal{B}_{T}^0:
\|z\|_{\mathcal{B}_{T}^0}^2\leq k\}, \qquad \text{ for some }
k\geq0,
$$
then $\mathcal{B}_k\subseteq \mathcal{B}_{T}^0$ is a bounded closed
convex set, and for $z\in \mathcal{B}_k$, we have
\begin{equation}\label{somme1}
\begin{array}{ll}
  \|z_t+y_t\|_{\mathcal{B}_h}^2& \leq 2(\|z_t\|^2_{\mathcal{B}_{h}}+\|y_t\|^2_{\mathcal{B}_{h}})
  \\\\
   &\leq 4(l^2\sup_{0\leq s\leq t}\E\|z(s)\|^2+\|z_0\|^2_{\mathcal{B}_{h}}\\\\
   &+l^2\sup_{0\leq s\leq
   t}\E\|y(s)\|^2+\|y_0\|^2_{\mathcal{B}_{h}})\\\\
   &\leq 4l^2(k+M^2\E\|\varphi(0)\|^2) +4\|y\|_{\mathcal{B}_h}^2\\\\
   &:=q'.
\end{array}
\end{equation}

It is clear that the operator $\Pi$ has a fixed point if and only if
$\widehat{\Pi}$ has one, so it turns to prove that $\widehat{\Pi}$
has a fixed point. To this end, we decompose $\widehat{\Pi}$ as
$\widehat{\Pi}=\Pi_1+\Pi_2$, where $\Pi_1$ and $\Pi_2$ are defined
on $\mathcal{B}_{T}^0$, respectively by
\begin{equation} (\Pi_1z)(t)=\left\{\begin{array}{ll}
&0\;\;\text {if } \; t\in(-\infty,0], \\\\
&g(t,z_t+y_t)-T_\alpha(t)g(0,\varphi ))+ \int_0^t(t-s)^{\alpha-1}
AS_\alpha(t-s)g(s,z_s+y_s)ds\\\\
&+\int_0^t(t-s)^{\alpha-1}
S_\alpha(t-s)\sigma(s)dB^H(s),\;\;\;\;\text {if } \; t\in[0,T],
\end{array}\right.
\end{equation}

and
 \begin{equation}
(\Pi_2z)(t)=\left\{\begin{array}{ll}
&0\;\;\text {if } \; t\in(-\infty,0], \\\\
& \int_0^t (t-s)^{\alpha-1}S_\alpha(t-s)f(s,z_s+y_s)ds\\\\
&+\int_0^t (t-s)^{\alpha-1}S_\alpha(t-s)Bu_{z+y}(s)ds,\;\text {if }
\; t\in[0,T].
\end{array}\right.
\end{equation}


For the sake of convenience,  the proof will be given in several
steps.

\ni{\bf {Step 1.}} We claim that there exists a positive number $k$,
such that $\Pi_1(x)+\Pi_2(x)\in \mathcal{B}_k$ whenever $x\in
\mathcal{B}_k$. If it is not true, then for each positive number
$k$, there is a function $z^k(.)\in\mathcal{B}_k$, but
$\Pi_1(z^k)+\Pi_2(z^k)\notin \mathcal{B}_k$, that is
$\E\|\Pi_1(z^k)(t)+\Pi_2(z^k)(t)\|^2>k$ for some $t\in [0,T].$
However, on the other hand, we have
\begin{equation}\label{Es1}
\begin{array}{ll}
  k  <\E \|\Pi_1(z^k)(t)+\Pi_2(z^k)(t)\|^2& \leq 6\{\E
   \|T_\alpha(t)g(0,\varphi)\|^2+\E\|g(t,z^k_t+y_t)\|^2\\\\
   &+\E\|\int_0^t
   (t-s)^{\alpha-1}AS_\alpha(t-s)g(s,z^k_s+y_s)ds\|^2\\\\
   &+\E\|\int_0^t
   (t-s)^{\alpha-1}S_\alpha(t-s)f(s,z^k_s+y_s)ds\|^2\\\\
   &+\E\|\int_0^t (t-s)^{\alpha-1}S_\alpha(t-s)Bu_{z^k+y}(s)ds\|^2\\\\

   &+\E\|\int_0^t (t-s)^{\alpha-1}S_\alpha(t-s)\sigma(s)dB^H(s)\|^2 \}\\\\
   &\leq 6\sum_{i=1}^6I_i.
\end{array}
\end{equation}

By $(\mathcal{H}.3)$, $(i)$ of Lemma \ref{operators}, we have
\begin{equation}\label{Es2}
\begin{array}{ll}
  I_1 & \leq \E\|T\alpha(t)g(0,\varphi)\|^2\\\\
   & \leq M^2\|(-A)^{-\beta}\|^2\|(-A)^\beta g(0,\varphi)\|^2\\\\
   & \leq M^2 c_1^2M_g[\|\varphi\|^2_{\mathcal{B}_h}+1].
\end{array}
\end{equation}

By $(\mathcal{H}.3)$, (\ref{somme1}), we have
\begin{equation}\label{Es3}
\begin{array}{ll}
I_2&\leq \|(-A)^{-\beta}\|^2\E\|(-A)^\beta g(t,z^k_t+y_t)\|^2\\\\
&\leq c_1^2M_g[\|z^k_t+y_t\|_{\mathcal{B}_h}^2+1]\\\\
&\leq c_1^2M_g[ 4l^2(k+M^2\E\|\varphi(0)\|^2)
+4\|y\|_{\mathcal{B}_h}^2+1).
\end{array}
\end{equation}
By $(iv)$ of Lemma \ref{operators}, $(\mathcal{H}.3)$, H\"{o}lder
inequality, we have
\begin{equation}\label{Es4}
\begin{array}{ll}
I_3&\leq \E\|\int_0^t
   (t-s)^{\alpha-1}AS_\alpha(t-s)g(s,z^k_s+y_s)ds\|^2\\\\
   &\leq \E\|(\int_0^t
   (t-s)^{\alpha-1}(-A)^{1-\beta}S_\alpha(t-s)(-A)^{\beta}g(s,z^k_s+y_s)ds\|^2\\\\
   &\leq\E( \int_0^t
   (t-s)^{\alpha-1}\|(-A)^{1-\beta}S_\alpha(t-s)(-A)^{\beta}g(s,z^k_s+y_s)\|ds)^2\\\\
   &\leq\frac{\alpha^2M^2_{1-\beta}\Gamma^2(\beta+1)}{\Gamma^2(\alpha\beta+1)}\E( \int_0^t
   (t-s)^{\alpha-1}\|(t-s)^{\alpha\beta-\alpha}(-A)^{\beta}g(s,z^k_s+y_s)\|ds)^2\\\\
   &\leq\frac{\alpha^2M^2_{1-\beta}\Gamma^2(\beta+1)}{\Gamma^2(\alpha\beta+1)} \int_0^t
   (t-s)^{2\alpha\beta-2} ds \int_0^t \E\|(-A)^{\beta}g(s,z^k_s+y_s)\|^2ds\\\\
   &\leq \frac{
   T^{2\alpha\beta-1}\alpha^2M^2_{1-\beta}\Gamma^2(\beta+1)}{(2\alpha\beta-1)\Gamma^2(\alpha\beta+1)}\int_0^t
   M_g( 4l^2(k+M^2\E\|\varphi(0)\|^2) +4\|y\|_{\mathcal{B}_h}^2+1)ds\\\\
&\leq \frac{
   T^{2\alpha\beta}\alpha^2M^2_{1-\beta}\Gamma^2(\beta+1)}{(2\alpha\beta-1)\Gamma^2(\alpha\beta+1)}M_g[
   4l^2(k+M^2\E\|\varphi(0)\|^2) +4\|y\|_{\mathcal{B}_h}^2+1].
\end{array}
\end{equation}

From $(\mathcal{H}.2)$, H\"{o}lder inequality, we have

\begin{equation}\label{Es5}
\begin{array}{ll}
I_4&\leq \E\|\int_0^t
   (t-s)^{\alpha-1}S_\alpha(t-s)f(s,z^k_s+y_s)ds\|^2\\\\
   &\leq
   \frac{M^2T}{\Gamma^2(\alpha)}\E\int_0^t\|(t-s)^{\alpha-1}f(s,z^k_s+y_s)\|^2ds\\\\
   &\leq
   \frac{M^2T}{\Gamma^2(\alpha)}\int_0^T(T-s)^{2\alpha-2}\E\|f(s,z^k_s+y_s)\|^2ds\\\\
   & \leq
   \frac{M^2T}{\Gamma^2(\alpha)}\int_0^T(T-s)^{2\alpha-2}p(s)\vartheta(\|z^k_s+y_s\|^2_{\mathcal{B}_h})ds\\\\
   & \leq
   \frac{M^2T}{\Gamma^2(\alpha)}\vartheta(4l^2(k+M^2\E\|\varphi(0)\|^2) +4\|y\|_{\mathcal{B}_h}^2)\int_0^T(T-s)^{2\alpha-2}p(s) ds\\\\
\end{array}
\end{equation}

From $(ii)$ of $(\mathcal{H}.2)$, H\"{o}lder inequality, it follows
that for $\delta>\frac{1}{2\alpha-1}$,

$$
\begin{array}{ll}
\int_0^T(T-s)^{2\alpha-2}p(s) ds&\leq
\left(\int_0^T(T-s)^{\frac{(2\alpha-2)\delta}{\delta-1}}ds\right)^{\frac{\delta-1}{\delta}}\left(\int_0^T(p(s))^\delta
ds\right)^{\frac{1}{\delta}}\\\\
&\leq
T^{\frac{(2\alpha-1)\delta-1}{\delta}}\left(\int_0^T(p(s))^\delta
ds\right)^{\frac{1}{\delta}}\\\\
&<\infty.
\end{array}
$$

From our assumptions, $(iv)$ of Lemma \ref{operators}, using the
fact that $(\sum_{i=1}^n a_i)^2 \leq n \sum_{i=1}^n a_i^2 $ for any
positive real numbers $a_i$, $i=1,2,...,n,$  we have
\begin{equation}\label{estim0}
\begin{array}{ll}
  \E\|u_{z+y}\|^2 \leq&
  6M_w\{\|x_1\|^2+M^2\E\|\varphi(0)\|^2+M^2c_1^2M_g[\|y\|^2_{\mathcal{B}_h}+1]\\\\
   & +[c_1^2  +\frac{\alpha^2M_{1-\beta}^2T^{2\alpha\beta}\Gamma^2(\beta+1)}{(2\alpha\beta-1)\Gamma^2(\alpha\beta+1)}]M_g[4l^2(k+M^2\E\|\varphi(0)\|^2) +4\|y\|_{\mathcal{B}_h}^2
   +1]\\\\
   &+\frac{M^2}{\Gamma^2(\alpha)}  \vartheta(4l^2(k+M^2\E\|\varphi(0)\|^2) +4\|y\|_{\mathcal{B}_h}^2)\int_0^T(T-s)^{2\alpha-2}p(s)ds\\\\
   &+2\frac{M^2}{\Gamma^2(\alpha)}T^{2H-1}\int_0^T(T-s)^{(2\alpha-2)}\|\sigma(s)\|^2_{\mathcal{L}_2^0}ds\}:=\mathcal{G}.
\end{array}
\end{equation}
For $p>\frac{1}{2\alpha-1}$, we have

\begin{equation}\label{estimfbm}
\begin{array}{ll}
  \int_0^T(T-s)^{(2\alpha-2)}\|\sigma(s)\|^2_{\mathcal{L}_2^0}ds &\leq\left(\int_0^T(T-s)^{
  \frac{(2\alpha-2)p}{p-1}}ds\right)^{\frac{p-1}{p}}\left(\int_0^T\|\sigma(s)\|^{2p}_{\mathcal{L}_2^0}ds\right)^{\frac{1}{p}}\\\\
  & \leq
T^{\frac{(2\alpha-1)p-1}{p}}\left(\int_0^T\|\sigma(s)\|^{2p}_{\mathcal{L}_2^0}ds\right)^{\frac{1}{p}}
\\\\
   & <\infty.
\end{array}
\end{equation}

 By (\ref{estim0}), $(i)$ of Lemma \ref{operators}, H\"{o}lder
 inequality, we have
 \begin{equation}\label{Es6}
\begin{array}{ll}
  I_5 & \leq \E\|\int_0^t (t-s)^{\alpha-1}S_\alpha(t-s)Bu_{z^k+y}(s)ds\|^2\\\\

& \leq \frac{M^2M_b}{\Gamma^2(\alpha)}
\int_0^t(t-s)^{2\alpha-2}ds\int_0^t\E\|u_{z^k+y}(s)\|^2ds\\\\
   & \leq
   \frac{6M^2M_bM_wT^{2\alpha}}{(2\alpha-1)\Gamma^2(\alpha)}\{ \|x_1\|^2+M^2\E\|\varphi(0)\|^2+M^2c_1^2M_g[\|y\|^2_{\mathcal{B}_h}+1]\\\\
   & +[c_1^2  +\frac{\alpha^2M_{1-\beta}^2T^{2\alpha\beta}\Gamma^2(\beta+1)}{(2\alpha\beta-1)\Gamma^2(\alpha\beta+1)}]M_g[4l^2(k+M^2\E\|\varphi(0)\|^2) +4\|y\|_{\mathcal{B}_h}^2
   +1]\\\\
   &+\frac{M^2}{\Gamma^2(\alpha)}  \vartheta(4l^2(k+M^2\E\|\varphi(0)\|^2) +4\|y\|_{\mathcal{B}_h}^2)\int_0^T(T-s)^{2\alpha-2}p(s)ds\\\\
   &+2\frac{M^2}{\Gamma^2(\alpha)}T^{2H-1}\int_0^T(T-s)^{(2\alpha-2)}\|\sigma(s)\|^2_{\mathcal{L}_2^0}ds
   \}.
\end{array}
 \end{equation}

By Lemma \ref{lem2}, Lemma \ref{operators}, (\ref{estimfbm}), for
$p>\frac{1}{2\alpha-1}$, we have

\begin{equation}\label{Es7}
\begin{array}{ll}
  I_6 & \leq\E\|\int_0^t
  (t-s)^{\alpha-1}S_\alpha(t-s)\sigma(s)dB^H(s)\|^2\\\\
   & \leq \frac{2M^2T^{2H-1}}{\Gamma^2(\alpha)}\int_0^T(T-s)^{(2\alpha-2)}\|\sigma(s)\|^2_{\mathcal{L}_2^0}ds
   \\\\
   & \leq \frac{2M^2T^{2H-1}}{\Gamma^2(\alpha)}
   T^{\frac{(2\alpha-1)p-1}{p}}\left(\int_0^T\|\sigma(s)\|^{2p}_{\mathcal{L}_2^0}ds\right)^{\frac{1}{p}}.
\end{array}
 \end{equation}

By (\ref{Es1}),  (\ref{Es2}), (\ref{Es3}), (\ref{Es4}), (\ref{Es5}),
(\ref{Es6}), (\ref{Es7}), we have

$$
\begin{array}{ll}
  k  <&\E \|\Pi_1(z^k)(t)+\Pi_2(z^k)(t)\|^2\leq \overline{K}+ 24l^2k c_1^2M_g+24l^2k \frac{
   T^{2\alpha\beta}\alpha^2M^2_{1-\beta}\Gamma^2(\beta+1)}{(2\alpha\beta-1)\Gamma^2(\alpha\beta+1)}M_g\\\\
&+6(1+\frac{6M^2M_bM_wT^{2\alpha}}{(2\alpha-1)\Gamma^2(\alpha)})
\frac{M^2T}{\Gamma^2(\alpha)}\vartheta(4l^2(k+M^2\E\|\varphi(0)\|^2)
\\\\
&+4\|y\|_{\mathcal{B}_h}^2)\int_0^T(T-s)^{2\alpha-2}p(s)
ds\\\\
&+\frac{144M^2M_bM_wT^{2\alpha}}{(2\alpha-1)\Gamma^2(\alpha)}[c_1^2
+\frac{\alpha^2M_{1-\beta}^2T^{2\alpha\beta}\Gamma^2(\beta+1)}{(2\alpha\beta-1)\Gamma^2(\alpha\beta+1)}]M_gl^2k,
  \end{array}
  $$

 where
 {\small{
 $$
 \begin{array}{ll}
    \overline{K} & = 6M^2 c_1^2(M_g\|\varphi\|^2_{\mathcal{B}_h}+6c_1^2M_g\left[4l^2M^2\E\|\varphi(0)\|^2 +4\|y\|_{\mathcal{B}_h}^2+1\right]\\\\
    &
    +6 \frac{
   T^{2\alpha\beta}\alpha^2M^2_{1-\beta}\Gamma^2(\beta+1)}{(2\alpha\beta-1)\Gamma^2(\alpha\beta+1)}
   M_g\left[ 4l^2M^2\E\|\varphi(0)\|^2 +4\|y\|_{\mathcal{B}_h}^2+1\right]\\\\
   &+\frac{36M^2M_bM_wT^{2\alpha}}{(2\alpha-1)\Gamma^2(\alpha)}\{
   \|x_1\|^2+M^2\E\|\varphi(0)\|^2+M^2c_1^2M_g\left[\|y\|^2_{\mathcal{B}_h}+1\right]\\\\
   &+\frac{6M^2M_bM_wT^{2\alpha}}{(2\alpha-1)\Gamma^2(\alpha)}[c_1^2  +\frac{\alpha^2M_{1-\beta}^2T^{2\alpha\beta}
   \Gamma^2(\beta+1)}{(2\alpha\beta-1)\Gamma^2(\alpha\beta+1)}]
   M_g\left[4l^2M^2\E\|\varphi(0)\|^2 +4\|y\|_{\mathcal{B}_h}^2+1\right]
   \}\\\\
    &+6(1+\frac{6M^2M_bM_wT^{2\alpha}}{(2\alpha-1)\Gamma^2(\alpha)})\frac{2M^2T^{2H-1}}{\Gamma^2(\alpha)}
   T^{\frac{(2\alpha-1)p-1}{p}}\left(\int_0^T\|\sigma(s)\|^{2p}_{\mathcal{L}_2^0}ds\right)^{\frac{1}{p}}.
 \end{array}
 $$}}
 Noting that $\overline{K}$ is independent of $k$.  Dividing both sides by
 $k$ and taking the lower limit as $k\longrightarrow\infty$, we
 obtain
 $$
 q'=4l^2(k+M\E\|\varphi(0)\|^2)
 +4\|y\|_{\mathcal{B}_h}\longrightarrow\infty\;\text { as }
 k\longrightarrow\infty,
 $$
 $$
\liminf_{k\longrightarrow\infty}\frac{\vartheta(q')}{k}=\liminf_{k\longrightarrow\infty}\frac{\vartheta(q')}{q'}.\frac{q'}{k}=4l^2\gamma.
 $$
 Thus, we have
 $$
 \begin{array}{ll}
   1 & \leq 24l^2 c_1^2M_g+24l^2 \frac{
   T^{2\alpha\beta}\alpha^2M^2_{1-\beta}\Gamma^2(\beta+1)}{(2\alpha\beta-1)\Gamma^2(\alpha\beta+1)}M_g\\\\
&+24l^2\gamma(1+\frac{6M^2M_bM_wT^{2\alpha}}{(2\alpha-1)\Gamma^2(\alpha)})
\frac{M^2T}{\Gamma^2(\alpha)}\int_0^T(T-s)^{2\alpha-2}p(s)
ds\\\\
&+\frac{144M^2M_bM_wT^{2\alpha}}{(2\alpha-1)\Gamma^2(\alpha)}[c_1^2
+\frac{\alpha^2M_{1-\beta}^2T^{2\alpha\beta}\Gamma^2(\beta+1)}{(2\alpha\beta-1)\Gamma^2(\alpha\beta+1)}]M_gl^2.
 \end{array}
 $$

 This contradicts  (\ref{cond1}). Hence for some positive $k$,
 $$
(\Pi_1+\Pi_2) (\mathcal{B}_k)\subseteq \mathcal{B}_k.
 $$

\ni{\bf {Step 2.}} $\Pi_1$ is a contraction.\\
Let $t\in[0,T]$ and $z^1, z^2\in\mathcal{B}_{T}^0$
$$
 \begin{array}{ll}
             \E\|(\Pi_1z^1)(t)-(\Pi_1z^2)(t) \|^2& \leq 2\E\|g(t,z_t^1+y_t)-g(t,z_t^2+y_t)\|^2\\ \\
              &
              +2\E\|\int_0^t(t-s)^{\alpha-1}AS_\alpha(t-s)(g(s,z^1_s+y_s)-g(s,z^2_s+y_s))ds\|^2
              \\\\
              & \leq 2M_g\|(-A)^{-\beta}\|^2\|z_s^1-z_s^2\|^2_{\mathcal{B}_h}\\ \\
              &+2\int_0^t(t-s)^{\alpha-1}(-A)^{1-\beta}S_\alpha(t-s)(-A)^\beta(g(s,z^1_s+y_s)-g(s,z^2_s+y_s))ds\|^2
              \\\\
& \leq 2M_g\|(-A)^{-\beta}\|^2\|z_s^1-z_s^2\|^2_{\mathcal{B}_h}\\\\

&+\frac{2\alpha^2M^2_{1-\beta}\Gamma^2(\beta+1)}{\Gamma^2(\alpha\beta+1)}\int_0^t(t-s)^{2\alpha\beta-2}ds\int_0^tM_g
              \|z_s^1-z_s^2\|^2_{\mathcal{B}_h}ds\\\\
              & \leq  2M_g\left\{\|(-A)^{-\beta}\|^2+\frac{2\alpha^2M^2_{1-\beta}\Gamma^2(\beta+1)}{\Gamma^2(\alpha\beta+1)}\frac{T^{2\alpha\beta}}{2\alpha\beta-1}\right\}(2l^2\sup_{0\leq s\leq T}\\\\
              &\E\|z^1(s)-z^2(s)\|^2 +2(\|z^1_0\|^2_{\mathcal{B}_h}+\|z^2_0\|^2_{\mathcal{B}_h})
              \\\\
              & \leq \nu\sup_{0\leq s\leq T}\E\|z^1(s)-z^2(s)\|^2)
              \quad ( \text{ since }\;z^1_0=z^2_0=0)
           \end{array}
$$
Taking supremum over $t$,
$$
\|(\Pi_1z^1)(t)-(\Pi_1z^2)(t) \|_{\mathcal{B}_{T}^0}\leq
\nu\|z^1-z^2\|_{\mathcal{B}_{T}^0},
$$
where
$$
\nu=4M_gl^2\left\{c_1^2+\frac{2\alpha^2M^2_{1-\beta}\Gamma^2(\beta+1)}{\Gamma^2(\alpha\beta+1)}\frac{T^{2\alpha\beta}}{2\alpha\beta-1}\right\}.
$$
By  $(\mathcal{H}.6)$, we have $\nu<1$. Thus $\Pi_1$ is a
contraction on $\mathcal{B}_{T}^0$.

\ni{\bf {Step 3.}} $\Pi_2$ is completely continuous
$\mathcal{B}_{T}^0$.\\
\ni{\bf {Claim 1.}} $\Pi_2$ is  continuous on $\mathcal{B}_{T}^0$.\\
Let $z^n$ be a sequence such that $z^n\longrightarrow z$ in
  $\mathcal{B}_{T}^0$. Then,  for  $t\in[0,T]$,  and thanks to  hypothesis $(\mathcal{H}.2)-(\mathcal{H}.3)$, for
     each $t\in[0,T]$, we have
    $$
    f(t,z_t^n+y_t)\longrightarrow f(t,z_t+y_t),
    $$
    $$
    g(t,z_t^n+y_t)\longrightarrow g(t,z_t+y_t).
    $$

%

 By  the dominated convergence theorem, we obtain  continuity of $\Pi_2$
$$
\begin{array}{ll}
  \E\|\Pi_2z^n(t)-(\Pi_2z)(t)\|^2 & \leq 2\E\|\int_0^t(t-s)^{\alpha-1}S_\alpha(t-s)B[u_{z^n+y}-u_{z+y}]ds\|^2\\ \\
   & +2\E\|\int_0^t(t-s)^{\alpha-1}S_\alpha(t-s)[f(s,z_s^n+y_s)-f(s,z_s+y_s)]ds\|^2 \\\\
&\leq\frac{2M^2M_b}{\Gamma^2(\alpha+1)}\frac{T^{2\alpha-1}}{2\alpha-1}\int_0^T \E\| u_{z^n+y}(s)-u_{z+y}(s)\|^2ds\\\\
&+\frac{2M^2}{\Gamma^2(\alpha+1)}\frac{T^{2\alpha-1}}{2\alpha-1}\int_0^T\E\|f(s,z_s^n+y_s)-f(s,z_s+y_s)\|^2ds\\ \\
   &\longrightarrow 0 \text{ as } n\longrightarrow\infty.
\end{array}
$$
Thus, $\Pi_2$ is continuous.

 \ni{\bf {Claim 2.}} $\Pi_2$ maps  $\mathcal{B}_k$ into equicontinuous
 family.   Let $z\in \mathcal{B}_k$ and $|h|$  be sufficiently small,    we have
  $$
  \begin{array}{ll}
    \E\|&(\Pi_2z)(t+h)-(\Pi_2z)(t)\|^2 \leq  \E\|\int_0^{t+h}(t+h-s)^{\alpha-1}S_\alpha(t+h-s)Bu_{z+y}(s)ds
    \\\\
     & +\int_0^{t+h}(t+h-s)^{\alpha-1}S_\alpha(t+h-s)f(s,z_s+y_s)ds\\\\
     & -\int_0^{t}(t-s)^{\alpha-1}S_\alpha(t-s)Bu_{z+y}(s)ds\\\\
     &-\int_0^t(t-s)^{\alpha-1}S_\alpha(t-s)f(s,z_s+y_s)ds\|^2\\\\
     &\leq 6
     \E\|\int_0^{t}\left((t+h-s)^{\alpha-1}-(t-s)^{\alpha-1}\right)S_\alpha(t+h-s)Bu_{z+y}(s)ds\|^2\\\\
     &+6\E\|\int_t^{t+h}(t+h-s)^{\alpha-1}S_\alpha(t+h-s)Bu_{z+y}(s)ds\|^2\\\\
     &+6
     \E\|\int_0^{t}(t-s)^{\alpha-1}\left(S_\alpha(t+h-s)-S_\alpha(t-s)\right)Bu_{z+y}(s)ds\|^2\\\\
     &+6
     \E\|\int_0^{t}\left((t+h-s)^{\alpha-1}-(t-s)^{\alpha-1}\right)S_\alpha(t+h-s)f(s,z_s+y_s)ds\|^2\\\\
     &+6\E\|\int_t^{t+h}(t+h-s)^{\alpha-1}S_\alpha(t+h-s)f(s,z_s+y_s)ds\|^2\\\\
     &+6
     \E\|\int_0^{t}(t-s)^{\alpha-1}\left(S_\alpha(t+h-s)-S_\alpha(t-s)\right)f(s,z_s+y_s)ds\|^2.
  \end{array}
  $$
  From $(iii)$ of Lemma \ref{operators}, we have  $S_\alpha(t)$ is
  compact for any $t>0$. Let $0<\varepsilon<t<T$, and $\delta>0$ such that
  $\|S_\alpha(\tau_1)-S_\alpha(\tau_2)\|\leq\epsilon$ for every
  $\tau_1,\tau_2\in[0,T]$  with $|\tau_1-\tau_2|\leq\delta $.
From (\ref{estim0}), $(i)$  of Lemma \ref{operators}, H\"{o}lder
inequality, it follows that
\begin{equation}\label{equi1}
 \begin{array}{ll}
    &\E\|(\Pi_2z)(t+h)-(\Pi_2z)(t)\|^2 \\\\
    &\leq
    \frac{6M^2M_b \mathcal{G}T}{\Gamma^2(\alpha)}\int_0^{t}\left((t+h-s)^{\alpha-1}-(t-s)^{\alpha-1}\right)^2ds
    \\\\
     & +\frac{6M^2M_b\mathcal{G} h}{\Gamma^2(\alpha)}\int_t^{t+h}(t+h-s)^{2\alpha-2}ds \\\\
     &+\frac{6M^2T^{2\alpha}\mathcal{G}}{2\alpha-1} \epsilon\\\\
      &+\frac{6M^2T\vartheta(q')}{\Gamma^2(\alpha)}\int_0^t\left((t+h-s)^{\alpha-1}-(t-s)^{\alpha-1}\right)^2p(s)ds\\\\
      &+\frac{6M^2T\vartheta(q')}{\Gamma^2(\alpha)}\int_t^{t+h}(t+h-s)^{2(\alpha-1)}p(s)ds\\\\
      &+\frac{6M^2T}{2\alpha-1}\epsilon\int_0^{t}(t-s)^{2(\alpha-1)}p(s)ds.\\
  \end{array}
  \end{equation}

  From $(ii)$ of $(\mathcal{H}.2)$, H\"{o}lder inequality, it follows
that for $\delta>\frac{1}{2\alpha-1}$,

$$
\begin{array}{ll}
\int_0^t(t-s)^{2\alpha-2}p(s) ds&\leq
\left(\int_0^t(t-s)^{\frac{(2\alpha-2)\delta}{\delta-1}}ds\right)^{\frac{\delta-1}{\delta}}\left(\int_0^T(p(s))^\delta
ds\right)^{\frac{1}{\delta}}\\\\
&\leq
T^{\frac{(2\alpha-1)\delta-1}{\delta}}\left(\int_0^T(p(s))^\delta
ds\right)^{\frac{1}{\delta}}\\\\
&<\infty.
\end{array}
$$
Similarly, we have
$$
\int_0^{t}(t+h-s)^{2(\alpha-1)}p(s)ds<\infty.
$$
By the dominated convergence theorem, we have
$$
\int_0^t\left((t+h-s)^{\alpha-1}-(t-s)^{\alpha-1}\right)^2p(s)ds\longrightarrow
0, \text{ as } h\longrightarrow0.
$$
   Therefore, for sufficiently small positive number $\epsilon$, we
   have from
   (\ref{equi1})  that
   $$
\E\|(\Pi_2z)(t+h)-(\Pi_2z)(t)\|^2\longrightarrow0\text{ as }
h\longrightarrow0.
   $$

   Thus, $\Pi_2$ maps $\mathcal{B}_k$ into an equicontinuous family
  of   functions.

  \ni{\bf {Claim 3.}} $(\Pi_2 \mathcal{B}_k)(t)$ is precompact set in $X$.\\
  Let $0<t\leq T$ be fixed,   and $\epsilon$ be a number satisfying  $0<\epsilon<t$. For $\delta>0$  and $z\in \mathcal{B}_k$,  we define

  $$
  \begin{array}{ll}
     (\Pi^\delta_{2,\epsilon}z)(t)&=
     \alpha\int_0^{t-\epsilon}\int_\delta^{\infty}\theta(t-s)^{\alpha-1}\eta_\alpha(\theta)S((t-s)^\alpha\theta)f(s,z_s+y_s)d\theta ds\\\\ \\
     &
     +\alpha\int_0^{t-\epsilon}\int_\delta^{\infty}\theta(t-s)^{\alpha-1}\eta_\alpha(\theta)S((t-s)^\alpha\theta)Bu_{z+y}(s)d\theta ds\\\\
     &=S(\epsilon^\alpha\delta)\alpha\int_0^{t-\epsilon}\int_\delta^{\infty}\theta(t-s)^{\alpha-1}\eta_\alpha(\theta)S((t-s)^\alpha\theta-\epsilon^\alpha\delta)
     f(s,z_s+y_s)d\theta ds\\\\ \\
     &+
     S(\epsilon^\alpha\delta)\alpha\int_0^{t-\epsilon}\int_\delta^{\infty}\theta(t-s)^{\alpha-1}\eta_\alpha(\theta)S((t-s)^\alpha\theta-\epsilon^\alpha\delta)
     Bu_{z+y}(s)d\theta ds
  \end{array}
  $$
From the compactness of $S(t)$ $(t>0)$, we obtain that the set
$V^\delta_\epsilon(t)=\{(\Pi^\delta_{2,\epsilon}z)(t):\; z\in
\mathcal{B}_k\}$ is relative compact in $X$ for every $\epsilon$,
$0<\epsilon<t$ and $\delta>0$. Moreover, for every $z\in
\mathcal{B}_k$, we have
\begin{equation}\label{prec1}
\begin{array}{ll}
    \E\|\Pi_{2}z)(t)-&\Pi^\delta_{2,\epsilon}z)(t)\|^2\leq
    4\alpha^2\E\|\int_0^{t}\int_0^\delta\theta(t-s)^{\alpha-1}\eta_\alpha(\theta)S((t-s)^\alpha\theta)f(s,z_s+y_s)d\theta ds\|^2
    \\\\
   & +4\alpha^2\E\|\int^t_{t-\epsilon}\int_\delta^{\infty}\theta(t-s)^{\alpha-1}\eta_\alpha(\theta)S((t-s)^\alpha\theta)f(s,z_s+y_s)d\theta ds\|^2\\\\
   &+4\alpha^2\E\|\int_0^{t}\int_0^\delta\theta(t-s)^{\alpha-1}\eta_\alpha(\theta)S((t-s)^\alpha\theta)Bu_{z+y}(s)d\theta ds\|^2\\\\
   &+4\alpha^2\E\|\int^t_{t-\epsilon}\int_\delta^{\infty}\theta(t-s)^{\alpha-1}\eta_\alpha(\theta)S((t-s)^\alpha\theta)Bu_{z+y}(s)d\theta ds\|^2\\\\
   &=4\sum_{i=1}^4J_i.

   \end{array}
   \end{equation}

A similar argument as before, we can show that
\begin{equation}\label{j1}
\begin{array}{ll}
  J_1 & \leq \alpha^2M^2T\E\int_0^t\|\int_0^\delta\theta(t-s)^{\alpha-1}\eta_\alpha(\theta)f(s,z_s+y_s)d\theta\|^2ds
  \\\\
   & \leq  \alpha^2M^2T
   \|\int_0^\delta\theta\eta_\alpha(\theta)d\theta\|^2\int_0^t(t-s)^{2\alpha-2}\E\|f(s,z_s+y_s)\|^2ds\\\\
   & \leq \alpha^2M^2T
   \vartheta(q')\|\int_0^\delta\theta\eta_\alpha(\theta)d\theta\|^2\int_0^t(t-s)^{2\alpha-2}p(s)ds.
\end{array}
\end{equation}
For $J_2$,  by (\ref{prob1}),  we have
\begin{equation}\label{j2}
\begin{array}{ll}
  J_2 &  \leq \alpha^2M^2T
   \vartheta(q')\|\int_0^\infty\theta\eta_\alpha(\theta)d\theta\|^2\int_{t-\epsilon}^t(t-s)^{2\alpha-2}p(s)ds\\\\
   &\leq \frac{\alpha^2M^2T
   \vartheta(q')}{\Gamma^2(1+\alpha)}\int_{t-\epsilon}^t(t-s)^{2\alpha-2}p(s)ds\\\\
   &\leq \frac{\alpha^2M^2T
   \vartheta(q')}{\Gamma^2(1+\alpha)}\left(\int_{t-\epsilon}^t(t-s)^{\frac{(2\alpha-2)\delta}{\delta-1}}ds\right)^{\frac{\delta-1}{\delta}}\left(\int_{t-\epsilon}^t(p(s))^\delta
ds\right)^{\frac{1}{\delta}}\\\\
&\leq \frac{\alpha^2M^2T
   \vartheta(q')}{\Gamma^2(1+\alpha)}\epsilon^{\frac{(2\alpha-1)\delta-1}{\delta}}\left(\int_{t-\epsilon}^t(p(s))^\delta
ds\right)^{\frac{1}{\delta}},
\end{array}
\end{equation}

where $\delta>\frac{1}{2\alpha-1}$.

For $J_3$, by H\"{o}lder inequality,  we have
\begin{equation}\label{j3}
\begin{array}{ll}
  J_3 &  \leq \alpha^2 \E\left(\int_0^{t}\int_0^\delta\|\theta(t-s)^{\alpha-1}\eta_\alpha(\theta)S((t-s)^\alpha\theta)Bu_{z+y}(s)\|d\theta ds\right)^2\\\\
&\leq
\alpha^2M^2M_bT\int_0^t(t-s)^{2\alpha-2}\E\|u_{z+y}(s)\|^2ds\|\int_0^\delta
\theta\eta_\alpha(\theta)d\theta\|^2.
  \end{array}
\end{equation}

For $J_4$, by (\ref{prob1}), we have
\begin{equation}\label{j4}
 \begin{array}{ll}
  J_4 &  \leq \alpha^2M^2\E\int_{t-\epsilon}^{t}\|(t-s)^{\alpha-1}Bu_{z+y}(s)\|^2ds\int_{t-\epsilon}^{t}\|\int_0^\infty \theta\eta_\alpha(\theta)d\theta\|^2ds \\\\
&\leq\frac{\epsilon\alpha^2M^2M_b}{\Gamma^2(\alpha+1)}\int_{t-\epsilon}^{t}(t-s)^{2\alpha-2}\E\|u_{z+y}(s)\|^2ds
  \end{array}
\end{equation}

 Put (\ref{j1}), (\ref{j2}), (\ref{j3}), (\ref{j4}) into
 (\ref{prec1}) to obtain
$$
\E\|\Pi_{2}z)(t)-\Pi^\delta_{2,\epsilon}z)(t)\|^2\longrightarrow0,\qquad
\text{ as } \epsilon\longrightarrow 0^+,\, \delta\longrightarrow0^+.
$$
Therefore,  there are  precompact sets arbitrarily close to the set
$V(t)=\{(\Pi_2z)(t):\; z\in B_k\}$, hence the set $V(t)$ is also
precompact in $X$.

Thus, by Arzela-Ascoli theorem $\Pi_2$ is a compact operator. These
arguments enable us to conclude that $\Pi_2$  is completely
continuous, and by the fixed point theorem of Karasnoselskii there
exists a fixed point $z(.)$ for $\widehat{\Pi}$ on $\mathcal{B}_k$.
If we define $x(t)=z(t)+y(t),$ $-\infty<t\leq T$, it is easy to see
that $x(.)$ is a mild solution  of (\ref{eq1}) satisfying
$x_0=\varphi$, $x(T)=x_1$. Then the proof is complete.

\end{proof}

\section{Example}

 To illustrate the previous result, we consider the following   fractional neutral stochastic partial differential equation with infinite
 delays,
driven by a fractional Brownian motion of the form\\
   \small{
   \begin{equation}\label{eq1part}
 \left\{\begin{array}{llll}
 dJ_t^{1-\alpha}[v(t,\xi)-g(t,v(t-r,\xi))-\varphi(0,\xi)+g(0,v(-r,\xi))]=[\frac{\partial^2}{\partial^2\xi}
v(t,\xi)+c(\xi)u(t)\\\\
+f(t,t-r,\xi)]dt
 +\sigma (t)\frac{dB^H(t)}{dt},\quad 0\leq t\leq T,\,r>0,\, \, 0\leq \xi\leq1\\\\
v(t,0)=v(t,1)=0,\quad \quad 0\leq t\leq T,\\\\
v(s,\xi)=\varphi(s,\xi) ,\,\;;-\infty< s \leq 0\quad 0\leq \xi\leq1,
\end{array}\right.
\end{equation}
} where   $B^H(t)$ is  cylindrical  fractional Brownian motion,
$\varphi: (-\infty,0]\times[0,1]\longrightarrow\R$ is a given
measurable and satisfies $\|\varphi\|^2_{\mathcal{B}_h}<\infty.$

We rewrite (\ref{eq1part}) into abstract form of (\ref{eq1}). We
take  $X=Y=U=L^2([0,1])$. Define the operator $A:D(A)\subset
X\longrightarrow X$ given  by $A=\frac{\partial^2}{\partial^2\xi}$
with
$$
D(A)=\{y\in X:\,y' \mbox{ is absolutely continuous},  y''\in X,\quad
y(0)=y(1)=0\},
$$
then we get
$$
Ax=\sum_{n=1}^\infty n^2<x,e_n>_Xe_n,\quad x\in D(A),
$$
 where $
e_n:=\sqrt{\frac{2}{\pi}}\sin nx,\; n=1,2,.... $
 is  an orthogonal  set of eigenvector of $-A$.\\

 The bounded linear operator $(-A)^{\frac{2}{3}}$ is given by
 $$
(-A)^{\frac{2}{3}}x=\sum_{n=1}^\infty
n^\frac{4}{3}<x,e_n>_Xe_n,\quad
 $$
 with domain
 $$
 D((-A)^{\frac{2}{3}})=\{x\in X,\sum_{n=1}^\infty
n^\frac{4}{3}<x,e_n>_Xe_n\in X \}.
$$
 It is  known that $A$  generates  a compact
analytic semigroup $\{S(t)\}_{t\geq 0}$ in $X$, and is given by (see
\cite{pazy})
$$ S(t)x=\sum_{n=1}^{\infty}e^{-n^2t}<x,e_n>e_n,
$$
for $x\in X$ and $t\geq0$.  Since the semigroup $\{S(t)\}_{t\geq 0}$
is analytic, there exists a constant $M>0$  such that
$\|S(t)\|^2\leq M$ for every $t\geq0$. In other words, the condition
$(\mathcal{H}.1)$ holds.

If we choose $\alpha\in(\frac34,1)$,
$$
S_\alpha(t)x=\int_0^\infty\alpha\theta\eta_\alpha(\theta)S(\theta
t^\alpha) d\theta, \quad x\in X.
 $$

  \ni Further,   the
operator $B: \R\longrightarrow X$ is a bounded linear operator
defined by
  $
Bu(t)(\xi)=c(\xi)u(t),\;0\leq \xi \leq1, \, c(\xi)\in L^2([0,1]),
  $ and    the operator $W:L^2([0,T],
U)\longrightarrow X $ is given by
  $$
W u(\xi)=\int_0^T(T-s)^{\alpha-1}S_\alpha(T-s)c(\xi)u(t)ds,\;\;0\leq
\xi \leq1,
  $$
  $W$ is linear and by H\"{o}lder inequality, we can show that $W$ is bounded  operator but not necessarily one-to-one.
  Let
  $$Ker \,W=\{x\in L^2([0,T],U),  \; W x=0\}
  $$
   be the null space of
  $W$ and $[Ker\, W]^\bot$ be its orthogonal complement in
  $L^2([0,T],U)$. Let $\widetilde{W}: [Ker\, W]^\bot\longrightarrow Range(W)
  $ be the restriction of $W$ to $[Ker\, W]^\bot$, $\widetilde{W}$ is
  necessarily one-to-one operator. The inverse mapping theorem  says that
  $\widetilde{W}^{-1}$ is bounded since $[Ker\, W]^\bot$ and $Range
  (W)$ are Banach spaces. So that  $W^{-1}$ is bounded and takes values in  $L^2([0,T],U)\setminus Ker\, W$,  hypothesis  $(\mathcal{H}.5)$ is satisfied.

We choose the phase function $h(s)=e^{2s}$, $s<0$, then
$l=\int_{-\infty}^0h(s)ds=\frac{1}{2}<\infty$, and the abstract
phase space $\mathcal{B}_h$ is Banach space with the norm
 $$
\|\varphi\|_{\mathcal{B}_h}=\int_{-\infty}^0h(s)\sup_{\theta\in[s,0]}(\E\|\varphi(\theta)\|^2)^\frac{1}{2}ds.
$$

To rewrite the initial-boundary value problem (\ref{eq1part}) in the
abstract form (\ref{eq1}), we assume the following:

For $(t,\varphi)\in[0,T]\times\mathcal{B}_h$, where
$\varphi(\theta)(\xi)=\varphi(\theta,\xi),$
$(\theta,\xi)\in(-\infty,0]\times[0,1]$, we put
$v(t)(\xi)=v(t,\xi)$. Define $g:[0,T]\times
\mathcal{B}_h\longrightarrow X$, $f:[0,T]\times
\mathcal{B}_h\longrightarrow X$ by

\begin{align*}
 (-A)^{\frac23}g(t,\varphi)(\xi)&=\int_{-\infty}^0e^{-4\theta}\varphi(\theta)(\xi)d\theta,\\\\
  f(t,\varphi)(\xi)&=
  \int_{-\infty}^0\mu(t,\xi,\theta)f_1(\varphi(\theta)(\xi))d\theta,
\end{align*}

where
\begin{itemize}
  \item [(i)] the function $\mu(t,\xi,\theta)\geq0$ is continuous
  in $[0,T]\times[0,1]\times(-\infty,0)$,
  $$
\int_{-\infty}^0\mu(t,\xi,\theta)d\theta=p_1(t,\xi)<\infty,
\qquad\text{ and  } \qquad
\left(\int_0^1p_1^2(t,\xi)\right)\frac12=p(t)<\infty;
  $$
  \item [(ii)] the function $f_1(.)$ is continuous, $0\leq f_1(v(\theta,\xi))\leq
  \vartheta(\|v(\theta,.)\|_{L^2})$ for
  $(\theta,\xi)\in(-\infty,0)\times(0,1)$, where $\vartheta(.):
  [0,\infty)\longrightarrow(0,\infty)$ is continuous and
  nondecreasing.
\end{itemize}

 By the similar method as in Balasubramaniyam and Ntouyas
 \cite{nto06}, we can show that the assumptions
 $(\mathcal{H}.2)-(\mathcal{H}.3)$  are satisfied.

In order to define the operator $Q: Y:=L^2([0,1],\R)\longrightarrow
Y$, we choose a sequence $\{\lambda_n\}_{n\in\N}\subset \R^+$, set
$Qe_n=\lambda_ne_n$, and assume that
$$
tr(Q)=\sum_{n=1}^\infty \sqrt{\lambda_n}<\infty.
$$
 Define the fractional Brownian motion in $Y$ by
$$
B^H(t)=\sum_{n=1}^\infty \sqrt{\lambda_n}\beta^H(t)e_n,
$$
where $H\in(\frac{1}{2},1)$ and $\{\beta^H_n\}_{n\in\N}$ is a
sequence of one-dimensional fractional Brownian motions mutually
independent. Let us assume the  function
$\sigma:[0,+\infty)\rightarrow
\mathcal{L}_2^0(L^2([0,1]),L^2([0,1]))$ satisfies
 $$\int_0^T\|\sigma(s)\|^{2p}_{\mathcal{L}_2^0}ds< \infty,\;\; \text {for some }  p>\frac{1}{2\alpha-1}.
  $$

Then all the assumptions of Theorem \ref{th1} are satisfied.
Therefore, we conclude that the system (\ref{eq1part}) is
controllable on $(-\infty,T]$.





\begin{thebibliography}{30}
\bibitem {ahm09} H.M. Ahmed. On some fractional stochastic
integrodifferential equations in Hilbert spaces. International
Journal of Mathematics and Mathematical 2009, 2009, DOI
10.1155/2009/568078. Article ID 568078, 8 pages.

\bibitem {nto06}{P. Balasubramaniyam and S.K. Ntouyas}
Controllability for neutral stochastic functional differential
inclusion with infinite delay in abstract space. J. Math. anal appl.
324 (2006), 161-176.
\bibitem  {biagini08}{F. Biagini, Y. Hu, B. {\O}ksendal, and T. Zhang, }
Stochastic Calculus for Fractional Brownian Motion and Application.
Springer-Verlag, (2008).




\bibitem  {boufoussi2}
{B. Boufoussi,  S. Hajji, and  E. Lakhel.  } {Functional
differential equations in Hilbert spaces driven by a fractional
Brownian motion}. {Afrika Matematika}, 23 (2) (2012), 173-194.

\bibitem  {boufoussi3}
{B. Boufoussi and  S.  Hajji}. {Neutral stochastic functional
differential equation driven by a  fractional  Brownian motion in a
Hilbert space},  \textit{Statist. Probab. Lett.}, 82 (2012),
1549-1558.


\bibitem  {carab}
{ T. Caraballo ,   MJ.  Garrido-Atienza, and   T. Taniguchi.} {The
existence and exponential behavior of solutions to stochastic delay
evolution equations with a fractional Brownian motion},
\textit{Nonlinear Analysis}, 74  (2011),  3671-3684.

\bibitem{iRL}{R. Coelho and  L.  Decreusefond.} {
Video correlated traffic models communications networks. } In
Proceedings to the ITC Seminar  on Telegraphic management (1995).

\bibitem  {cui12} J. Cui and L. Yan. Existence result for
fractional  neutral stochastic integrodifferential equations with
infinite delay. Journal of Physics A: Mathematical and Theoretical
44 (2011), 1-16.

   \bibitem  {dung15}
 {NT. Dung}.   { Stochstic Volterra integro-differential equations driven by
    by fractional Brownian motion in Hilbert space}, \textit{Stochastics},  87 (1)  (2015), 142-159.

\bibitem {bor06}MM. El-Bori. On some stochastic fractional
integrodifferential equations. Advances in Dynamical Systems and
Applications 1 (2006), 49-57.


     \bibitem{kil06}AA. Kilbas,   H.M.  Srivastava and JJ.
     Trujillo. Theory and applications of fractional differential
     equations. Elsevier, Amsterdam (2006).


\bibitem  {klam07}{J. Klamka.}  Stochastic controllability of linear systems
with delay in control, \textit{Bull. Pol. Acad. Sci. Tech. Sci.}, 55
 (2007), 23-29.

\bibitem  {klam13}{J. Klamka.} Controllability of dynamical systems. A survey.
\textit{Bull. Pol. Acad. Sci. Tech. Sci.},  61   (2013), 221-229.
\bibitem{iLL} {J. R. Le\'{o}n   and  C. Lunden\~{a}.} {  Estimating the
 diffusion coefficient for diffusions driven by fBm}.  Stat.  inference and Stoch.
 Prcess (2000).


\bibitem  {lak16} { E. Lakhel.} Controllability of Neutral Stochastic Functional
Integro-Differential Equations Driven By Fractional Brownian Motion.
 Stochastic Analysis and Applications (To appear).

\bibitem  {lak15} { E. Lakhel and S. Hajji.  } Existence and Uniqueness   of Mild Solutions to  Neutral
SFDEs driven by a  Fractional  Brownian Motion  with Non-Lipschitz
Coefficients. Journal of Numerical Mathematics and Stochastics, 7
(1) (2015), 14-29.

\bibitem  {lak8} E. Lakhel   \,and M. A.  McKibben.  Controllability of Impulsive
Neutral Stochastic  Functional Integro-Differential Equations Driven
by Fractional Brownian Motion. Chapter 8 In book : Brownian Motion:
Elements, Dynamics, and Applications.  Editors:  M. A. McKibben \&
M. Webster.  Nova Science Publishers, New York, 2015, pp. 131-148.

\bibitem  {li13} K. Li. Stochastic delay fractional evolution
equations driven by fractional brownian motion. Mathematical Methods
in the Applied Sciences 2014. DOI 10. 1002/mma. 3169.
\bibitem  {li-li} {Y. Li and  B. Liu}.  Existence of solution of
nonlinear neutral functional differential inclusion with infinite
delay. Stoc. Anal. Appl. 25 (2007), 397-415.




\bibitem{MV68} {\textsc  B. Mandelbrot  and  V. Ness.}{  Fractional Brownian
motion, fractional noises and applications.} SIAM Reviews, {\textbf
10}(4) (1986), 422-437.




\bibitem  {pazy}
 { A. Pazy}. {Semigroups of Linear Operators and Applications to Partial Differential Equations.
   Applied Mathematical Sciences, vol. 44},  Springer-Verlag, New York (1983).

\bibitem  {ren13}
{Y. Ren,  X. Cheng,  and R. Sakthivel.
  } On time-dependent stochastic evolution equations driven by
fractional Brownian motion in Hilbert space with finite delay.
 {Mathematical methods in the Applied Sciences}, 37  (2013),
2177-2184.

\bibitem {ren11} {Y. Ren, L. Hu, and R. Sakthivel.} Controllability of impulsive neutral
stochastic functional differential inclusions with infinite delay,
 {J. Comput. Appl. Math.}, 235 (8)  (2011), 2603-2614.

\bibitem{ren13a}{R. Sakthivel, R. Ganesh, Y. Ren, and  S. M. Anthoni.} Approximate controllability of nonlinear fractional
dynamical systems.  {Commun. Nonlinear Sci. Numer. Simul.}, 18
  (2013), 3498-3508.

\bibitem  {sak13a}
{R. Sakthivel, P. Revathi and Y. Ren.} Existence of  solutions for
nonlinear  fractional stochastic differential   equations. Nonlinear
Anal. 81 (2013), 70-86.

\bibitem  {sak13b}
{R. Sakthivel, P. Revathi and NI. Mahmodov.} asymptotic stability of
fractional stochastic neutral differential equations with infinite
delays. Abstract and Applied Anal. 2013 (2013), 1-9. Article ID
769257.
%
\bibitem  {zho10}  Y. Zhou, J.  Feng.  Existence of mild solutions for fractional neutral evolution
  equations. Comput. Math. Appl. 59 ( 2010), 1063-1077.

\bibitem  {zho14} Y. Zhou. Basic theory of fractional differential
equations. World Scientific Publishing Co. Pte. Ltd. (2014).
\bibitem  {zho12} Y. Zhou, J. Wang and M. Medved. On the solvability
and optimal controls of fractional integrodifferntial evolution
systems with infinite delay. J. Optim.  theory App. 152 (2012),
31-50.

\end{thebibliography}
\end{document}